\documentclass[11pt]{article}
\usepackage{amsthm, amsmath, amssymb,mathrsfs, enumerate, fullpage, mathtools, extpfeil}
\mathtoolsset{showonlyrefs=true} 
\usepackage{colonequals}
\usepackage{microtype}
\usepackage{tocloft}
\allowdisplaybreaks 

\usepackage{caption}
\usepackage{booktabs}
\usepackage{makecell}

\usepackage{color}
\usepackage{comment}
\definecolor{darkerGreen}{RGB}{0,195,0} 
\definecolor{darkerRed}{RGB}{240,0,0}
\usepackage[colorlinks=true, linkcolor=darkerRed, citecolor=darkerGreen]{hyperref}
\numberwithin{equation}{section}

\theoremstyle{plain}
\newtheorem{thm}{Theorem}[section]
\newtheorem*{theorem*}{Theorem}
\newtheorem*{question*}{Question}
\newtheorem{theorem}[thm]{Theorem}
\newtheorem{corollary}[thm]{Corollary}
\newtheorem{lemma}[thm]{Lemma}
\newtheorem{proposition}[thm]{Proposition}

\theoremstyle{remark} 
\newtheorem{remark}[thm]{Remark}

\newtheorem*{example*}{Example}

\theoremstyle{definition}
\newtheorem{definition}[thm]{Definition}
\newtheorem{assumption}[thm]{Assumption}

 \newcommand{\diag}{{\operatorname{diag}}}

\newcommand{\BN}{{\mathbb {N}}}

\newcommand{\BR}{{\mathbb {R}}}

\newcommand{\CB}{{\mathcal {B}}}

\newcommand{\CK}{{\mathcal {K}}}
\newcommand{\CL}{{\mathcal {L}}}

\newcommand{\CT}{{\mathcal {T}}}

\newcommand{\dd}{{\,\mathrm{d}}}
\newcommand{\Id}{\mathrm{Id}}

\usepackage[OT2,T1]{fontenc}
\DeclareSymbolFont{cyrletters}{OT2}{wncyr}{m}{n}
\DeclareMathSymbol{\Sha}{\mathalpha}{cyrletters}{"58}

\newcommand{\norm}[1]{\left\Vert #1\right\Vert}

\begin{document}

\title{Regularity of Second-Order Elliptic PDEs in Spectral Barron Spaces}
\author{Ziang Chen, Liqiang Huang, Mengxuan Yang, and Shengxuan Zhou}
\date{}
\maketitle

\begin{abstract}
We establish a regularity theorem for second-order elliptic PDEs on $\mathbb{R}^{d}$ in spectral Barron spaces. Under mild ellipticity and smallness assumptions, the solution gains two additional orders of Barron regularity. As a corollary, we identify a class of PDEs whose solutions can be approximated by two-layer neural networks with cosine activation functions, where the width of the neural network is independent of the spatial dimension.
 \end{abstract}

\section{Introduction}

Solving partial differential equations (PDEs) in high-dimensional spaces is a fundamental challenge in computational mathematics, primarily due to the rapid growth of computational complexity with respect to the spatial dimension $d$. For classical numerical methods, such as the finite element method, achieving an approximation accuracy of order $\varepsilon$ requires a computational cost of order $\mathcal{O}(\varepsilon^{-d})$. This phenomenon is commonly referred to as the \emph{curse of dimensionality} (CoD). Recently, a theoretical framework has been developed to address the CoD by approximating solutions of high-dimensional PDEs with two-layer neural networks
\begin{equation}
\label{eq:twonnw}
u_n(x) = \frac{1}{n} \sum_{i=1}^n a_i \sigma(w_i^{\top} x + b_i).
\end{equation}

One mathematical foundation underlying such approximation results is related to the so-called Barron spaces. 
In his seminal work~\cite{Barron93}, Barron proved that any function $g\colon\mathbb{R}^d\to\mathbb{R}$ whose Fourier transform $\widehat{g}$ satisfies
\begin{equation}
\label{eq:defbarronfunc}
C_g \coloneqq \int_{\mathbb{R}^d} |\widehat{g}(\xi)| \cdot |\xi| \,\mathrm{d}\xi < \infty,
\end{equation}
admits an $L^2$-approximation by a two-layer neural network of the form~\eqref{eq:twonnw} with accuracy $\varepsilon$, where the number of neurons $n$ can be bounded by
\[
\left\lceil C_g^2\varepsilon^{-2}\right\rceil.
\]
In particular, the approximation complexity depends on the Barron norm $C_g$ rather than explicitly on the spatial dimension $d$, i.e., it avoids exponential dependence on $d$ when $C_g$ is not. 
Following~\cite{Barron93}, numerous variants and generalizations of Barron-type function classes have been developed; see, for instance,~\cite{Jason18, E, EWoj22, SIEGELXu21, SX23, SW24}.

To take advantage of this strong approximation property in numerical PDEs, a natural strategy is the following: if one can prove that the exact solution $u^*$ of a given PDE belongs to a Barron space and that its Barron norm can be bounded independently of the spatial dimension, then Barron's theorem guarantees that $u^*$ admits a neural network approximation free from the CoD. This connection reduces the problem of the existence of a well-behaved two-layer neural network approximation for high-dimensional PDEs to the following fundamental analysis questions:
\begin{question*}
    Under what conditions do solutions of PDEs belong to a Barron space?
\end{question*}

Our goal is to establish conditions under which solutions of PDEs belong to Barron spaces with quantitatively controlled norms. To this end, we adopt a slightly modified definition of Barron functions compared to Barron's original definition~\eqref{eq:defbarronfunc}, focusing on the \emph{spectral Barron space} (with order $s$) defined as follows (cf.~Definition~\ref{def:spectral_Barron}):
\begin{equation}
	\label{eq:Bs}
	\mathcal{B}^s \coloneqq 
	\left\{ g \colon \mathbb{R}^d \to \mathbb{R} \;\middle|\;
	\|g\|_{\mathcal{B}^s}
	\coloneqq
	2^{s/2}
	\int_{\mathbb{R}^d}
	|\widehat{g}(\xi)|
	(1+|\xi|^2)^{s/2}
	\,\mathrm{d}\xi
	< \infty
	\right\}.
\end{equation}
Determining conditions under which solutions to specific PDEs live in these spectral Barron spaces is commonly referred to as the \emph{Barron regularity problem}.

Since second-order elliptic PDEs admit well-established Sobolev regularity results and have been extensively studied in the classical literature (e.g.,~\cite{Evans10}), it is natural to ask whether an analogous regularity theory can be developed in spectral Barron spaces. This direction was first explored in~\cite{LLWang2021}, where a Barron regularity theorem was established for the Schr\"{o}dinger equation on a bounded domain.
Subsequently,~\cite{CLLZ2023} obtained the first Barron regularity result for the Schr\"{o}dinger equation on $\mathbb{R}^d$. 
This approach was further extended in~\cite{FL25} to Schr\"{o}dinger-type equations with first-order terms, while~\cite{shuailu25} considered cases involving small Barron perturbations in the coefficients.
However, the existing Fourier-analytic frameworks primarily address equations in which the second-order operator has constant coefficients (for instance, $A(x)=I_d$ in~\eqref{eq:PDE}). 
As a result, these methods do not directly extend to general second-order elliptic PDEs with variable coefficients in the leading-order terms. 
The main mathematical difficulty arises from the variable-coefficient operator $\nabla \cdot (A(x)\nabla )$, which introduces convolution structures under the Fourier transform and prevents a direct application of the arguments developed in the aforementioned works.

\subsection*{Our Contribution}
In this work, we answer the above question by establishing a regularity theorem for general second-order elliptic PDEs of the form
\begin{equation}\label{eq:PDE}
-\nabla\cdot(A(x)\nabla u)+b(x)\cdot\nabla u+c(x)u=f(x)
\quad \text{on } \mathbb{R}^{d}.
\end{equation}
Our main result can be informally summarized as follows; precise statements are given in Theorem~\ref{thm:regularitythm} and Corollary~\ref{cor:dimind}.
\begin{theorem*}[Informal version]
Suppose that the coefficients of~\eqref{eq:PDE} satisfy assumptions (A1)--(A3) in Assumption~\ref{assum}. 
Then for any source term $f \in \mathcal{B}^{s}$, the unique solution $u^*$ to~\eqref{eq:PDE} gains two additional orders of Barron regularity, and satisfies the estimate
\[
\|u^*\|_{\mathcal{B}^{s+2}}
\leq C \|f\|_{\mathcal{B}^{s}},
\]
where the constant $C$ depends on the Barron norms of the coefficients, as well as on the dimension $d$ and the order $s$. 
Additionally, if Assumption (A3') is satisfied, then the constant $C$ can be made explicit and has no explicit dependence on the spatial dimension.
\end{theorem*}
We emphasize that the smallness assumption in~(A3) on the Barron norm of the second-order perturbation is not merely technical, but reflects an intrinsic requirement of the problem. 
As shown in Remark~\ref{rmk:mainrmk}, relaxing this condition may lead to the failure of ellipticity.

Combining the above Barron regularity estimate with the neural network approximation result in~\cite[Proposition~3.1]{FL25}, restated below as Theorem~\ref{thm:FL25},  we show that solutions to PDEs of the form~\eqref{eq:PDE} can be approximated by two-layer neural networks with cosine activation functions. Moreover, the number of neurons can be independent of the spatial dimension; see Corollary~\ref{cor:maincor} and Appendix~\ref{app:mainremark}. We use a simple but useful example to illustrate our result. 
\begin{example*}
Consider the elliptic equation
\[
-\nabla\cdot\bigl(A(x)\nabla u\bigr)
+ b(x)\cdot\nabla u
+ c(x)u
= \varphi(x)
\quad\text{on }\mathbb{R}^d,
\]
where \(\varphi\) is the standard Gaussian density function. Let
$
\mathbf{1}\coloneqq(1,1,\ldots,1)\in\mathbb{R}^d
$
denote the constant vector of ones, and let \(I_d\) denote the
\(d\times d\) identity matrix. Assume that
\[
\|A-I_d\|_{\mathcal{B}^{1}}
+\|b-\mathbf{1}\|_{\mathcal{B}^{0}}
+\|c-1\|_{\mathcal{B}^{0}}
\leq \frac{1}{2}.
\]
For any bounded domain $\Omega\subseteq \mathbb{R}^d$ with unit volume, our results guarantee that the solution can be approximated with accuracy $\varepsilon$ (in the $H^{2}$ sense) by a two-layer neural network of the form~\eqref{eq:twonnw} with cosine activation, with at most
\[
n = \lceil 36\varepsilon^{-2} \rceil
\]
neurons (cf.~Appendix~\ref{app:mainremark} for more details). Here, the explicit bound on the number of neurons does not depend on the spatial dimension $d$. Any implicit dimensional dependence is confined to the assumption that the Barron norms of the perturbations \(A-I_d\), \(b-\mathbf{1}\), and \(c-1\), as well as that of the source function \(\varphi\), are bounded independently of \(d\).
\end{example*}

Our proof builds on the Fourier-analytic framework developed in~\cite{CLLZ2023,FL25}. To address the difficulties arising from variable leading-order coefficients, we decompose \(A(x)\) into a constant matrix and a variable perturbation satisfying a quantitative spectral Barron norm bound, thereby separating the constant-coefficient contribution from the variable component. 
Within this setting, we apply the Banach fixed-point theorem in spectral Barron spaces and establish compactness of the associated operator via the Kolmogorov--Riesz theorem. 
In verifying the conditions of the Kolmogorov--Riesz theorem, we further isolate the contribution of the constant-coefficient part, which leads to a compactness argument different from those used in~\cite{CLLZ2023, FL25, shuailu25}.

A key ingredient in our analysis is the development of tools adapted to spectral Barron spaces.
Since $\mathcal{B}^{s}$ is a Banach space and lacks the Hilbert space structure of classical Sobolev spaces $H^{s}(\mathbb{R}^d)$, many standard PDE techniques are not directly applicable. 
Our approach resolves the convolution difficulty caused by variable coefficients while simultaneously overcoming the absence of Hilbert space structure, thereby extending the constant-coefficient regularity framework to a substantially broader class of elliptic operators.

\subsection*{Related Works}
Apart from our approach, which is based on establishing that the solution itself belongs to a Barron space, another line of work constructs approximate solutions to~\eqref{eq:PDE} from representation results. 
More precisely, if the coefficients are sufficiently regular so that the solution admits a Barron representation without exponential dependence on the dimension, then the solution can be efficiently approximated by a two-layer neural network. 
This idea was first realized in~\cite{Chen21}.  
Related questions for elliptic PDEs on bounded domains have also been studied recently. 
For instance,~\cite{Mar23} considered the case where the domain is a hypercube, while~\cite{CHuang25} further investigated Neumann boundary conditions on hypercubes.

\subsection*{Organization of the Paper}
\begin{itemize}
	\item In Section~\ref{sec:setup}, we introduce spectral Barron spaces and state our main regularity result, Theorem~\ref{thm:regularitythm} and Corollary~\ref{cor:dimind}, together with a sketch of the proof. We also present the neural network approximation result in Corollary~\ref{cor:maincor}.
	\item In Section~\ref{sec:proofs}, we establish the basic properties of spectral Barron spaces and use them to prove the regularity theorem.
    \item In Appendix~\ref{app}, we provide an alternative proof of Lemma~\ref{lem:mainlem}.
    \item In Appendix~\ref{app:mainremark}, we give an example to illustrate the dimension-independence of the neural network approximation in Corollary \ref{cor:maincor}.
\end{itemize}

\subsection*{Acknowledgments}
The authors are grateful to Ye Feng, Ryan Goh, Jianfeng Lu, Konstantinos Spiliopoulos, Qian Tang, Xuda Ye, and Caiwei Zhang for helpful discussions, and acknowledge the Princeton Machine Learning Theory Summer School 2025 for providing collaboration opportunities. The work of Z.~Chen is supported in part by the NSF grant DMS-2509011. M.~Yang acknowledges the support of AMS-Simons travel grant, NSF grant DMS-2554813, and the partial support of DARPA grant AIQ-HR001124S0029. The work of S.~Zhou is supported by LabEx CIMI.

\section{Main Results and Outline of the Proof}
\label{sec:setup}
Throughout this paper, for any real-valued function $ g $ defined on $\mathbb{R}^d$, 
we denote by $ \widehat{g} $ its Fourier transform, defined as
\begin{equation}
\label{eq:Fourier}
\widehat{g}(\xi)
\coloneqq \frac{1}{(2\pi)^d} \int_{\mathbb{R}^d} g(x) e^{-i x^{\top} \xi} \, \mathrm{d}x.
\end{equation}
The inverse Fourier transform is defined by
$$
g^{\vee}(x)
\coloneqq \int_{\mathbb{R}^d} g(\xi) e^{i x^{\top} \xi} \, \mathrm{d}\xi.
$$

\begin{definition}
\label{def:spectral_Barron}
Let $ g \colon \mathbb{R}^d \to \mathbb{R} $ be a real-valued function such that $ \widehat{g} \in L^1(\mathbb{R}^d) $, 
and let $ s \ge 0 $.  
The \emph{spectral Barron norm} of $ g $ is defined by
\begin{equation}
\label{eq:defbarronnorm}
\norm{g}_{\mathcal{B}^s} \coloneqq 2^{s/2}\int_{\mathbb{R}^d} |\widehat{g}(\xi)| (1+|\xi|^2)^{s/2} \dd\xi .
\end{equation}
The corresponding \emph{spectral Barron space} is given by
\begin{equation*}
\mathcal{B}^s \coloneqq 
\left\{ g \colon \mathbb{R}^d \to \mathbb{R} \; \middle| \; \widehat{g} \in L^1(\mathbb{R}^d), \; \norm{g}_{\mathcal{B}^s} < \infty \right\}.
\end{equation*}
\end{definition}
We choose the normalization constant $2^{s/2}$ so that the spectral Barron space $\mathcal{B}^s$ forms a Banach algebra (cf.~Proposition~\ref{prop:mainprop}). This choice is made solely to simplify the analysis.

Different variants of Barron-type norms have been developed in the literature 
since Barron's original work.  
For example,~\cite{SIEGELXu21, YMing22, Mar23, CLLZ2023, Ming25, FL25, shuailu25} 
adopt a definition consistent with Barron's original formulation, 
where the Barron norm is defined via the Fourier transform;  
these are commonly referred to as the \emph{spectral Barron norms}.  
In this paper, we follow this Fourier-based definition.  
In contrast, several alternative formulations avoid the use of the Fourier transform.  
For instance, some are based on probability measures 
(e.g.,~\cite{Chen21, E, SX23}), 
while others are defined in terms of the eigenfunctions of the Dirichlet eigenvalue problem 
(e.g.,~\cite{LuLu22, Mar23, CHuang25}).  
These are typically referred to simply as the \emph{Barron norms}.

Our result is based on the following assumptions.

\begin{assumption}
\label{assum}
The coefficients in equation~\eqref{eq:PDE} satisfy
\begin{enumerate}
	\item[(A1)] The zeroth-order coefficient satisfies $ c(x) \ge 0 $ and can be written as
\[
c(x) = \alpha + w(x),
\]
where $ \alpha > 0 $ and $ w(x) \in \mathcal{B}^s $ for some fixed $ s \ge 0 $.
\item[(A2)] The first-order coefficient $b(x)$ can be written as  
\[
b(x) = \beta + v(x),
\]
where $\beta \in \mathbb{R}^d$ is a constant vector and 
$v(x) = (v_1(x),\ldots,v_d(x))$ with each $v_i(x) \in \mathcal{B}^s$.
\item[(A3)] The symmetric matrix-valued coefficient \(A(x)\) admits a decomposition 
$$
A(x) = M + E(x),
$$
where $M $ is a constant symmetric matrix satisfying
\[
\xi^{\top} M \xi \ge m |\xi|^2 
\quad \text{for some } m > 0 \text{ and all } \xi \in \mathbb{R}^d,
\]
and $ E(x) = (e_{ij}(x))_{1 \le i,j \le d} $ is a symmetric matrix-valued perturbation with $ e_{ij}(x) \in \mathcal{B}^{s+1} $. Moreover,
\begin{equation}\label{eq:B1B2}
\|E(x)\|_{\CB^{s+1}} \coloneqq 
\sum_{i,j=1}^{d} 
\|e_{ij}(x)\|_{\mathcal{B}^{s+1}}<\min\{\alpha,m\}.
\end{equation}
\item[(A3')]
Same as (A3), except that~\eqref{eq:B1B2} is replaced by
\begin{equation}
    \label{eq:B1B2new}
     \|E(x)\|_{\CB^{s+1}}+ \|w(x)\|_{\CB^s} + \|v(x)\|_{\CB^s}\leq K < \min\{\alpha,m\},
\end{equation}
where
\begin{equation}
    \label{eq:newv}
\|v(x)\|_{\CB^s} \coloneqq \sum_{i=1}^{d} \|v_i(x)\|_{\CB^s}.
\end{equation}
\end{enumerate}
\end{assumption}

\begin{remark}
\label{rmk:mainrmk}
One motivation for condition~\eqref{eq:B1B2} is that the quantitative bound on the Barron norm of \(E(x)\) ensures the uniform ellipticity of \(A(x)\), a standard assumption for second-order elliptic PDEs. Indeed, for any function \(g\in \mathcal{B}^{s+1}\), the Fourier inversion formula gives
\begin{equation*}
\label{eq:pointwise-Barron}
|g(x)|
=\left|\int_{\BR^d}\widehat g(\xi)e^{ix^{\top}\xi}\dd\xi\right|
\le \int_{\BR^d}|\widehat g(\xi)|\dd\xi
=\norm{g}_{\mathcal B^0}
\le 2^{-(s+1)/2}\norm{g}_{\mathcal B^{s+1}},
\end{equation*}
for all $x\in\BR^d$. 
Consequently,
\[
|\xi^{\top}E(x)\xi|
\leq \sum_{i,j=1}^d
|e_{ij}(x)|\,|\xi_i|\,|\xi_j|
\leq 2^{-(s+1)/2}\norm{E}_{\mathcal B^{s+1}}|\xi|^2.
\]
It follows that
\[
\xi^\top A(x)\xi
=\xi^\top M\xi+\xi^\top E(x)\xi
\geq
\left(
m-2^{-(s+1)/2}\|E\|_{\mathcal{B}^{s+1}}
\right)|\xi|^2.
\]
Since condition~\eqref{eq:B1B2} implies
\(\|E\|_{\mathcal{B}^{s+1}}<m\), the quantity in parentheses is strictly positive. Therefore, \(A(x)\) is uniformly elliptic.

\end{remark}

We now state our regularity theorem, followed by its dimension-independent version presented as a corollary.
\begin{theorem}
	\label{thm:regularitythm}
Suppose that (A1)--(A3) in Assumption~\ref{assum} hold with $s\geq 0$.
Then, for any $ f \in \mathcal{B}^s $,  there exists a unique solution $ u^* \in \mathcal{B}^{s+2} $ to~\eqref{eq:PDE}.  
Moreover, the solution satisfies the estimate
\begin{equation}
\label{eq:regesti}
\|u^*\|_{\mathcal{B}^{s+2}} \le C \|f\|_{\mathcal{B}^s},
\end{equation}
where $C>0$ is the constant in~\eqref{eq:elliptic-bound},
depending only on the quantities $\alpha$, $\beta$, $M$, 
the Barron norms of $E(x)$, $w(x)$, and $v(x)$, 
the order $s$ in Assumption~\ref{assum}, 
and the spatial dimension $d$.
\end{theorem}

\begin{corollary}
    \label{cor:dimind}
    Suppose that (A1), (A2), and (A3') in Assumption~\ref{assum} hold with $s\geq 0$.
Then, for any $ f \in \mathcal{B}^s $,  there exists a unique solution $ u^* \in \mathcal{B}^{s+2} $ to~\eqref{eq:PDE}.  
Moreover, the solution satisfies the estimate
\begin{align}
\|u^*\|_{\mathcal{B}^{s+2}} &\le \frac{2(\min\{\alpha,m\}-\|E\|_{\CB^{s+1}}+\|w\|_{\CB^s}+\|v\|_{\CB^s})}
{(\min\{\alpha,m\}-\|E\|_{\CB^{s+1}})(\min\{\alpha,m\}-\|E\|_{\CB^{s+1}}-\|w\|_{\CB^s}-\|v\|_{\CB^s})} \|f\|_{\mathcal{B}^s}  \label{eq:regestinew} \\
&\leq \frac{2(\min\{\alpha,m\}+K)}
 {\min\{\alpha,m\}(\min\{\alpha,m\}-K)} \|f\|_{\mathcal{B}^s}.  \label{eq:regestinew2}
\end{align}
\end{corollary}

We sketch the outline of the proof of Theorem~\ref{thm:regularitythm} below. 
The detailed arguments will be presented in Section~\ref{sec:proofs}.
\begin{enumerate}[Step 1.]
	\item We apply the Banach fixed-point theorem together with the Banach algebra property of spectral Barron spaces to establish the regularity of 
	$ (\alpha+ \beta \cdot \nabla - \nabla \cdot A(x) \nabla)u = f $.
	
	\item The solution of~\eqref{eq:PDE} can be written as 
	\[
	u = (I + \mathcal{T})^{-1} (\alpha +\beta \cdot \nabla- \nabla \cdot A(x) \nabla)^{-1} f,
	\]
	where $ \mathcal{T}(u) \coloneqq (\alpha+\beta \cdot \nabla - \nabla \cdot A(x) \nabla)^{-1} (w u + v\cdot\nabla u) $.
	By Step~1, it suffices to show that the operator $ I + \mathcal{T} $ has a bounded inverse.
	
	\item We prove that $ \mathcal{T} $ is compact by the Kolmogorov--Riesz theorem, hence $ I + \mathcal{T} $ is a Fredholm operator. 
	Its bounded invertibility then reduces to injectivity, i.e., the equation~\eqref{eq:PDE} admits only the zero solution in spectral Barron spaces when $ f = 0 $, which follows from the weak maximum principle. 
\end{enumerate}
Although the above strategy follows the same general framework as in~\cite{CLLZ2023, FL25}, 
the analysis requires more technical effort due to the appearance of $ A(x) $ in the second-derivative term. 
We summarize below the main difficulties and how we overcome them:
\begin{enumerate}[(1)]
	\item The regularity of 
$ (\alpha+\beta \cdot \nabla - \nabla \cdot A(x) \nabla)u = f $ 
cannot be obtained directly via the Fourier transform, 
as in the constant-coefficient case, 
because variable coefficients lead to convolution terms. 
To address this, we decompose $A(x)$ into a constant matrix $M$ and a small perturbation $E(x)$; 
the part associated with $M$ is handled explicitly by the inverse Fourier transform.
	\item When proving that $ \mathcal{T} $ is compact, we cannot directly apply the Fourier transform as in~\cite{CLLZ2023, FL25}, for the same reason mentioned above. 
Instead, we replace $ A(x) $ by the constant matrix $ M $ and analyze the difference between $ \mathcal{T}(u) $ and the corresponding operator associated with $ M $ (cf.~\eqref{eq:AxM}).
\end{enumerate}

The mild dependence on the spatial dimension of two-layer neural network approximation for spectral Barron functions was established in~\cite{xu2020finite, CLLZ2023, FL25}, which we recall below.
\begin{theorem}[{\cite[Proposition~3.1]{FL25}}]
\label{thm:FL25}
Suppose $ u \in \mathcal{B}^k $ for some $ k \in \mathbb{N} $, 
and let $ \Omega \subseteq \mathbb{R}^d $ be a bounded domain with Lebesgue measure $ |\Omega| $.  
Then, for any $ n \in \mathbb{N}_{+} $,  
there exists a two-layer neural network of the form 
\begin{equation}
    \label{eq:neural}
u_n(x) = \frac{1}{n} \sum_{i=1}^n a_i \cos(w_i^{\top} x + b_i),
\end{equation}
where $ a_i, b_i \in \mathbb{R} $ and $ w_i \in \mathbb{R}^d $, such that
\[
\|u_n - u\|_{H^k(\Omega)}
\le \frac{\sqrt{|\Omega|}  \|u\|_{\mathcal{B}^k}}{\sqrt{n}}.
\]
\end{theorem}

Combining Theorem~\ref{thm:regularitythm} and Theorem~\ref{thm:FL25}, 
we obtain the following corollary, 
which roughly states that if all the functions appearing in~\eqref{eq:PDE} 
are spectral Barron functions, 
then the complexity of approximating the solution by a two-layer neural network is of order $ \mathcal{O}(\varepsilon^{-2}) $.

\begin{corollary}\label{cor:maincor}
Suppose that (A1)--(A3) in Assumption~\ref{assum} hold with
$s = k \in \mathbb{N}$, and that $f \in \mathcal{B}^k$.
Let $ u^* $ be the unique solution to~\eqref{eq:PDE} in $ \mathcal{B}^{k+2} $, and let $ C $ be the constant in Theorem~\ref{thm:regularitythm}.  
Then, for any bounded domain $ \Omega \subseteq \mathbb{R}^d $  
and any $ n \in \mathbb{N}_{+} $,  
there exists a two-layer neural network of the form~\eqref{eq:neural}
such that
\[
\|u_n - u^*\|_{H^{k+2}(\Omega)}
\le \frac{C \sqrt{|\Omega|}  \|f\|_{\mathcal{B}^k}}{\sqrt{n}}.
\]
In particular, in the sense of $H^{k+2}(\Omega)$-norm, the solution $u^*$ can be approximated (with error $\leq \varepsilon $) by a two-layer neural network of the
form~\eqref{eq:neural} with the number of neurons 
$$ n \leq  \left\lceil C^2 |\Omega| \|f\|_{\CB^k}^2 \varepsilon^{-2} \right\rceil.$$  
Additionally, if (A3') in Assumption~\ref{assum} holds, then the constant \(C\) can be improved to that in \eqref{eq:regestinew2}, which is independent of the spatial dimension.
\end{corollary}

In Appendix~\ref{app:mainremark}, we provide an example to further explain the ``dimension-independent'' computational complexity stated in the above corollary. 

\begin{remark}
The Barron regularity result is independent of the activation function. The cosine activation in Corollary~\ref{cor:maincor} arises only from Theorem~\ref{thm:FL25}. An analogous dimension-independent approximation theorem for another activation function would yield a corresponding conclusion; see, for example,~\cite[Theorem~3]{SIEGELXu21} for the case of a ReLU activation function.
\end{remark}

\section{Proof of the Regularity Theorem}
\label{sec:proofs}
In this section, we first study the analytical properties of spectral Barron spaces. 
Then we present the three-step proof in detail, following the outline given in Section~\ref{sec:setup}.

\subsection{Basic Properties of Spectral Barron Spaces}
\begin{proposition}
\label{prop:mainprop}
	The spectral Barron space $ \mathcal{B}^s $ with $ s \ge 0 $ satisfies the following properties:
	\begin{enumerate}[(1)]   
    \item $\mathcal{B}^s$ is a Banach algebra. Moreover, for any $ t \ge s $, we have $ \mathcal{B}^{t} \subseteq \mathcal{B}^s $ and 
\[
\|g\|_{\mathcal{B}^{s}} \le 2^{(s - t)/2}\,\|g\|_{\mathcal{B}^{t}} 
\quad \text{for all } g \in \mathcal{B}^{t}.
\]
    
    \item $
\mathcal{B}^s \subseteq C_b^{\lfloor s \rfloor}(\mathbb{R}^d),
$
the space of $ \lfloor s \rfloor $-times continuously differentiable and bounded functions on $ \mathbb{R}^d $,
where $ \lfloor \cdot \rfloor $ is the Gauss floor function.
Moreover, for any $ g \in \mathcal{B}^s $ and any multi-index 
$ \alpha = (j_1, \ldots, j_d) $ with $ |\alpha| = j_1 + \cdots + j_d \le s $,
we have $ \partial^{\alpha} g \in \mathcal{B}^{s - |\alpha|} $ and the estimate
\begin{equation}\label{eq:pfest2}
    \|\partial^{\alpha} g\|_{\mathcal{B}^{s - |\alpha|}} 
    \le 2^{-|\alpha|/2}\|g\|_{\mathcal{B}^{s}}.
\end{equation}
	\end{enumerate}
\end{proposition}
\begin{proof}
\noindent\text{(1)} 
The only nontrivial part in verifying that $ \mathcal{B}^s $ is a Banach space  
is the verification of its completeness.   
Suppose that $ \{ g_n \in \mathcal{B}^s \}_{n \ge 1} $ is a Cauchy sequence in $\CB^s$.  
Then the sequence 
\[
\left\{\widehat{g_n}(\xi) (1 + |\xi|^2)^{s/2} \right\}_{n \ge 1}
\]
is Cauchy in $ L^1(\mathbb{R}^d) $.  
Thus $ \widehat{g_n}(\xi) (1 + |\xi|^2)^{s/2} $ converges to some $ h(\xi) \in L^1(\mathbb{R}^d) $ by the completeness of $ L^1(\mathbb{R}^d) $.
It is straightforward to verify that $ \{ g_n \}_{n \ge 1} $ converges in $ \mathcal{B}^s $ to
\[
\left( h(\xi) (1 + |\xi|^2)^{-s/2} \right)^{\vee} \in \CB^s.
\]

To verify the product inequality required for the Banach algebra property,  
namely, for all $ g, h \in \mathcal{B}^s $,
\begin{equation}
    \label{eq:multi}
\|gh\|_{\mathcal{B}^s} \le \|g\|_{\mathcal{B}^s} \|h\|_{\mathcal{B}^s},
\end{equation}
we follow the same argument as in~\cite[Lemma~3.4]{CLLZ2023}.

The monotonicity identity follows directly from the definition.

\medskip
\noindent\text{(2)} 
Suppose $g \in \mathcal{B}^s$, then we have
$$
\left|
\widehat{g}(\xi)\, i^{|\alpha|}\xi_1^{j_1} \cdots \xi_d^{j_d} e^{i x^{\top} \xi}
\right|
\leq  \left|
\widehat{g}(\xi)\right| \cdot  |\xi|^{|\alpha|}
$$ 
and
\begin{equation}\label{eq:preLDC}
\int_{\mathbb{R}^d} 
\left|
\widehat{g}(\xi)\right| \cdot |\xi|^{|\alpha|}
\, \mathrm{d}\xi
\le 
\|g\|_{\mathcal{B}^s}
< \infty.
\end{equation} 
Hence, by the Lebesgue dominated convergence theorem, 
\begin{equation}\label{eq:paf}
\partial^{\alpha} g(x)
= 
\int_{\mathbb{R}^d} 
\widehat{g}(\xi)\, i^{|\alpha|}\xi_1^{j_1}\cdots\xi_d^{j_d} e^{i x^{\top} \xi}
\, \mathrm{d}\xi.
\end{equation}
Thus, \(\partial^\alpha g\) is continuous and bounded on
\(\mathbb{R}^d\) by~\eqref{eq:preLDC}. Finally,
estimate~\eqref{eq:pfest2} follows from
\begin{equation*}
	\begin{split}
\|\partial^{\alpha} g\|_{\mathcal{B}^{s-|\alpha|}}
=& 2^{(s-|\alpha|)/2}\int_{\mathbb{R}^d}
\left|\xi_1^{j_1}\cdots\xi_d^{j_d} \widehat{g}(\xi)\right|
(1+|\xi|^2)^{(s-|\alpha|)/2}\,\mathrm{d}\xi \\
\le &
2^{(s-|\alpha|)/2}\int_{\mathbb{R}^d}
|\xi|^{|\alpha|}
|\widehat{g}(\xi)|
(1+|\xi|^2)^{(s-|\alpha|)/2}\,\mathrm{d}\xi
\le 2^{-|\alpha|/2}\|g\|_{\mathcal{B}^{s}}. 
\end{split}
\end{equation*}
This completes the proof. 
\end{proof}

\subsection{Proof of Theorem~\ref{thm:regularitythm} and Corollary~\ref{cor:dimind}}

We begin by elaborating on Step~1 in the proof outline given in Section~\ref{sec:setup}. 
We present two approaches. 
The first is based on the Banach fixed point theorem, while the second relies on a perturbation argument, which we defer to Appendix~\ref{app}.
\begin{lemma}
\label{lem:mainlem}
Assume that condition~(A3) in Assumption~\ref{assum}
holds for some $s \ge 0$ and $\alpha > 0$,
and that $\beta \in \mathbb{R}^d$ is a constant vector.
Then for every $ f \in \mathcal{B}^s $,  
there exists a unique $ u \in \mathcal{B}^{s+2} $ satisfying
\begin{equation}
\label{eq:rewrite}
\left( \alpha+\beta \cdot \nabla - \nabla \cdot A(x) \nabla \right) u = f.
\end{equation}
Moreover, letting
\begin{equation}
    \label{eq:defL}
    L \coloneqq \frac{2}{\min\{ \alpha, m \} - \|E\|_{\CB^{s+1}}},
\end{equation}
we have the following Barron norm estimate:
\begin{equation*}
\left\| u \right\|_{\mathcal{B}^{s+2}}
\leq L \left\| f \right\|_{\mathcal{B}^s}.
\end{equation*}
\end{lemma}
\begin{proof}
We define the mapping
\[
\Phi \colon \mathcal{B}^{s+2} \longrightarrow \mathcal{B}^{s+2}, \quad 
u \longmapsto
( \alpha+\beta \cdot \nabla - \nabla \cdot M \nabla )^{-1} 
( f + \nabla \cdot E(x) \nabla u ).
\]
It is straightforward to verify from the Fourier multiplier representation and Proposition~\ref{prop:mainprop} that \(\Phi\) is well defined; that is, \(\Phi\) maps \(\mathcal{B}^{s+2}\) into itself.
We next show that \(\Phi\) is a contraction. 

By the definition of $\Phi$, for every $u_1, u_2 \in \mathcal{B}^{s+2}$ we have
\[
\left( \alpha+\beta \cdot \nabla - \nabla \cdot M \nabla \right)\left( \Phi(u_1) - \Phi(u_2) \right)
= \nabla \cdot E(x) \nabla (u_1 - u_2).
\]
Taking the $\mathcal{B}^s$-norm on both sides and applying Proposition~\ref{prop:mainprop}, we obtain
\begin{equation}
\label{eq:contraction}
\begin{split}
&\left\| \left( \alpha+\beta \cdot \nabla - \nabla \cdot M \nabla \right)\left( \Phi(u_1) - \Phi(u_2) \right) \right\|_{\mathcal{B}^s}
= \left\| \nabla \cdot E(x) \nabla (u_1 - u_2) \right\|_{\mathcal{B}^s} \\
=& \left\| \sum_{i=1}^{d}\partial_i\left( \sum_{j=1}^{d}e_{ij}\partial_j(u_1-u_2)\right)\right\|_{\CB^s} \leq \sum_{i=1}^{d}2^{-1/2}\left\| \sum_{j=1}^{d}e_{ij}\partial_j(u_1-u_2)\right\|_{\CB^{s+1}} \\
\leq & \sum_{i,j=1}^{d}2^{-1/2}\|e_{ij}\|_{\CB^{s+1}}\|\partial_j(u_1-u_2)\|_{\CB^{s+1}} \leq \sum_{i,j=1}^{d}2^{-1}\|e_{ij}\|_{\CB^{s+1}}\|u_1-u_2\|_{\CB^{s+2}}   \\
= & \frac{\|E\|_{\CB^{s+1}}}{2}  \|u_1 - u_2\|_{\mathcal{B}^{s+2}}.
\end{split}
\end{equation}
On the other hand, we note that
\begin{equation}
\label{eq:BstoBs+2}
\begin{split}
&\left\| \left( \alpha +\beta \cdot \nabla- \nabla \cdot M \nabla \right)\left( \Phi(u_1) - \Phi(u_2) \right) \right\|_{\mathcal{B}^s} \\
=& 2^{s/2}\int_{\mathbb{R}^d}
\left| \alpha +i\xi^{\top}\beta+ \xi^{\top} M \xi \right|\cdot 
\left| \widehat{\Phi(u_1)} - \widehat{\Phi(u_2)} \right|
(1 + |\xi|^2)^{s/2} \, \mathrm{d}\xi \\
\ge & \min\{\alpha, m\}2^{s/2}
\int_{\mathbb{R}^d} (1 + |\xi|^2)
\left| \widehat{\Phi(u_1)} - \widehat{\Phi(u_2)} \right|
(1 + |\xi|^2)^{s/2} \, \mathrm{d}\xi \\
=& \frac{\min\{\alpha, m\}}{2}
\|\Phi(u_1) - \Phi(u_2)\|_{\mathcal{B}^{s+2}}.
\end{split}
\end{equation}
Combining~\eqref{eq:contraction} and~\eqref{eq:BstoBs+2}, we obtain
\[
\|\Phi(u_1) - \Phi(u_2)\|_{\mathcal{B}^{s+2}}
\le
\frac{\|E\|_{\CB^{s+1}}}{\min\{\alpha, m\}}
\|u_1 - u_2\|_{\mathcal{B}^{s+2}}.
\]
Thus, by assumption~\eqref{eq:B1B2}, the mapping $\Phi$ is a contraction.
By the Banach fixed-point theorem, there exists a unique 
$u \in \mathcal{B}^{s+2}$ such that $\Phi(u) = u$, which is equivalent to
\[
\left( \alpha +\beta \cdot \nabla- \nabla \cdot A(x) \nabla \right) u = f.
\]

For the Barron norm estimate, we use the decomposition
\[
\left( \alpha+\beta \cdot \nabla - \nabla \cdot M \nabla \right) u
= f + \nabla \cdot E(x) \nabla u.
\]
Taking the $\mathcal{B}^s$-norm on both sides and applying the same estimates as in~\eqref{eq:contraction} and~\eqref{eq:BstoBs+2}, we obtain
\[
\frac{\min\{\alpha, m\}}{2}  \|u\|_{\mathcal{B}^{s+2}}
\le \|f\|_{\mathcal{B}^s}
+ \frac{\|E\|_{\CB^{s+1}}}{2}   \|u\|_{\mathcal{B}^{s+2}}.
\]
The desired result then follows directly.
\end{proof}

We next explain the details of Step~2 and Step~3 in the proof outline as stated in Section~\ref{sec:setup}. Motivated by  \cite{LLWang2021,CLLZ, FL25}, we rewrite equation~\eqref{eq:PDE} as an integral equation of the second kind:
\[
u + \mathcal{T}(u) = (\alpha+\beta \cdot \nabla - \nabla \cdot A(x) \nabla)^{-1} f,
\]
where
\begin{equation}
	\label{eq:opT}
\mathcal{T}(u) \coloneqq (\alpha +\beta \cdot \nabla- \nabla \cdot A(x) \nabla)^{-1} (w(x)u+v(x)\cdot \nabla u).
\end{equation}

\begin{lemma}
\label{lem:compactT}
Suppose that (A1)--(A3) in Assumption~\ref{assum} hold with $s\geq 0$.  
Then
the operator
\[
\mathcal{T} \colon \mathcal{B}^{s+1} \longrightarrow \mathcal{B}^{s+1},
\]
defined in~\eqref{eq:opT}, is compact.
\end{lemma}

The well-posedness of $\mathcal{T}$ is guaranteed by Proposition~\ref{prop:mainprop} and Lemma~\ref{lem:mainlem}.
To prove that $\mathcal{T}$ is compact, it suffices to show that the image of the closed unit ball in $ \mathcal{B}^{s+1} $,
\begin{equation}\label{TX}
\left\{ \mathcal{T}(u) : \| u \|_{\mathcal{B}^{s+1}} \leq 1 \right\},
\end{equation}
is relatively compact in $ \mathcal{B}^{s+1} $. Since $ \mathcal{B}^{s+1} $ is complete, relative compactness is equivalent to total boundedness. Therefore, it suffices to prove the total boundedness of the set
\begin{equation}\label{funct_class_F}
\mathcal{F} \coloneqq \left\{ \widehat{\mathcal{T}(u)}(\xi)  (1 + |\xi|^2)^{(s+1)/2} : \| u \|_{\mathcal{B}^{s+1}} \leq 1 \right\} \subseteq L^1(\mathbb{R}^d),
\end{equation}
where the $ \mathcal{B}^s $-norm is interpreted in terms of the standard $ L^1 $-norm,  
and we omit the normalization constant $ 2^{(s+1)/2} $ for simplicity,  
as scaling does not affect total boundedness.
The following Kolmogorov--Riesz theorem will be instrumental in establishing total boundedness.

\begin{theorem}[Kolmogorov--Riesz Theorem {\cite[Theorem 5]{Hanche2010}}]\label{thm:KRthm}
Let $ 1\leq p < \infty $. A subset $ \mathcal{F} \subseteq L^p(\mathbb{R}^d) $ is totally bounded if and only if the following three conditions hold:
\begin{enumerate}[(1)]
  \item $ \mathcal{F} $ is bounded in $ L^p(\mathbb{R}^d) $. That is,
\[
\sup_{f \in \mathcal{F}} \| f \|_{L^p(\mathbb{R}^d)} < \infty;
\]
  \item For every $ \varepsilon > 0 $, there exists $ R > 0 $ such that
  \[
  \int_{|x| > R} |f(x)|^p \dd x < \varepsilon^p, \quad \text{for all } f \in \mathcal{F};
  \]
  \item $ \mathcal{F} $ is uniformly equicontinuous. That is, for every $ \varepsilon > 0 $, there exists $ \delta > 0 $ such that
\[
\int_{\mathbb{R}^d} |f(x + y) - f(x)|^p \, \mathrm{d}x < \varepsilon^p, \quad \text{for all } f \in \mathcal{F} \text{ and all } | y | < \delta.
\]
\end{enumerate}
\end{theorem}

\begin{proof}[\textbf{Proof of Lemma~\ref{lem:compactT}}]
We verify the three conditions in Theorem~\ref{thm:KRthm} to establish the total boundedness of the set~\eqref{funct_class_F}.
Throughout the proof, we adopt the convention that 
\(\|v(x)\|_{\mathcal{B}^s}\) is defined as in~\eqref{eq:newv}.

\medskip
\noindent\textbf{Verification of condition (1).} 
Since $ \left( \alpha+\beta \cdot \nabla - \nabla \cdot A(x) \nabla \right) \mathcal{T}(u) = w(x) u+v(x)\cdot \nabla u $,  
by Lemma~\ref{lem:mainlem} and Proposition~\ref{prop:mainprop}, we have the following estimate for any $ \| u \|_{\mathcal{B}^{s+1}} \leq 1 $:
\begin{align*}
&\int_{\mathbb{R}^d} \left| \widehat{\mathcal{T}(u)}(\xi) \right| (1 + | \xi |^2)^{(s+1)/2} \, \mathrm{d}\xi
\leq \|\mathcal{T}(u)\|_{\CB^{s+2}} \\
\leq& L\left\| wu +v\cdot \nabla u\right\|_{\mathcal{B}^s} 
\leq  L (\| w \|_{\mathcal{B}^s} \| u \|_{\mathcal{B}^s}+\|v\|_{\CB^s}\|u\|_{\CB^{s+1}}) \\
\leq & L(\| w \|_{\mathcal{B}^s}+\|v\|_{\CB^s}).
\end{align*}
Thus, $ \mathcal{F} $ is bounded in $ L^1(\mathbb{R}^d) $.

\medskip
\noindent\textbf{Verification of condition (2).} 
For any $ \varepsilon > 0 $, there exists $ R > 0 $ such that
\[
L^{-1}(1+|\xi|^2)^{1/2} > \frac{1}{\varepsilon} \quad \text{for all } | \xi | > R.
\]
Then, for any $ u \in \mathcal{B}^{s+1} $ with $ \| u \|_{\mathcal{B}^{s+1}} \leq 1 $, it follows from Lemma~\ref{lem:mainlem} and Proposition~\ref{prop:mainprop} that
\begin{align*}
&\frac{1}{\varepsilon}\int_{| \xi | > R} \left| \widehat{\mathcal{T}(u)}(\xi) \right| (1 + | \xi |^2)^{(s+1)/2} \, \mathrm{d}\xi
\leq 
L^{-1}  \int_{| \xi | > R}  \left| \widehat{\mathcal{T}(u)}(\xi) \right| (1 + | \xi |^2)^{(s+2)/2} \, \mathrm{d}\xi \\
\leq&  L^{-1}\|\mathcal{T}(u)\|_{\CB^{s+2}} 
\leq \left\| w u +v\cdot \nabla u\right\|_{\mathcal{B}^s}
\leq \| w \|_{\mathcal{B}^s}+\|v\|_{\CB^s}.
\end{align*}
Thus, we obtain
\[
\int_{| \xi | > R} \left| \widehat{\mathcal{T}(u)}(\xi) \right| (1 + | \xi |^2)^{(s+1)/2} \, \mathrm{d}\xi
\leq \varepsilon (\| w \|_{\mathcal{B}^s}+\|v\|_{\CB^s}).
\]

\medskip
\noindent\textbf{Verification of condition (3).} 
Let $ F, G \in \mathcal{B}^{s+2} $ and $ H \in \mathcal{B}^s $ satisfy
\begin{equation}\label{eq:AxM}
\begin{split}
( \alpha+\beta \cdot \nabla - \nabla \cdot A(x) \nabla ) F &= w u + v \cdot \nabla u, \\
( \alpha+\beta \cdot \nabla - \nabla \cdot M \nabla ) G &= w u + v \cdot \nabla u, \\
( \alpha+\beta \cdot \nabla - \nabla \cdot M \nabla )(F - G) &= \nabla \cdot E(x) \nabla F =H.
\end{split}
\end{equation}
The existence of $ F $ is guaranteed by Lemma~\ref{lem:mainlem}, 
while the existence of $ G $ follows directly from taking the Fourier transform of the constant-coefficient equation.  
Taking Fourier transforms of the last two equations yields
\begin{equation}\label{eq:FGH}
\widehat{G}(\xi)
= \frac{\widehat{w u}(\xi) + \widehat{v \cdot \nabla u}(\xi)}{\alpha +i\xi^{\top}\beta+ \xi^{\top} M \xi},
\quad
\widehat{F - G}(\xi)
= \frac{\widehat{H}(\xi)}{\alpha +i\xi^{\top}\beta+ \xi^{\top} M \xi}.
\end{equation}

For a given $ \varepsilon > 0 $,  
we introduce the following constants and auxiliary functions used in the proof.
\begin{itemize}
\item By condition~(2), there exists $ R > 0 $ such that for all $ u \in \mathcal{B}^{s+1} $ with $ \|u\|_{\mathcal{B}^{s+1}} \le 1 $,
\begin{equation}
\label{eq:xiR}
\int_{|\xi| > R} 
|\widehat{F}(\xi)|(1 + |\xi|^2)^{(s+1)/2}
\, \mathrm{d}\xi
< \varepsilon.
\end{equation}
	\item 
Define
\[
C_1 \coloneqq \sup_{|\xi| \le 2R} 
\left| \frac{(1 + |\xi|^2)^{(s+1)/2}}{\alpha +i\xi^{\top}\beta+ \xi^{\top} M \xi}\right| ,
\quad
C_2 \coloneqq \sup_{|\xi| \le 3R} 
\left| L(\alpha +i\xi^{\top}\beta+ \xi^{\top} M \xi)\right|.
\]
Both are finite by the smoothness of the functions involved.  
Moreover, by uniform continuity, for every $ \varepsilon > 0 $ there exists $ 0 < \delta_1 < R $ such that
\[
\left|
\frac{(1 + |\xi|^2)^{(s+1)/2}}{\alpha +i\xi^{\top}\beta+ \xi^{\top} M \xi}
-
\frac{(1 + |\xi'|^2)^{(s+1)/2}}{\alpha +i\xi'^{\top}\beta+ \xi'^{\top} M \xi'}
\right|
< \varepsilon,
\quad
\text{for all } | \xi |, | \xi' | \le 3R
\text{ with } | \xi - \xi' | < \delta_1.
\]

\item For every $ 1 \le i,j \le d $, since $ \widehat{e_{ij}} \in L^1(\mathbb{R}^d) $,  
there exists $ \varphi_{ij} \in C_c^{\infty}(\mathbb{R}^d) $ such that~\cite[Corollary~4.23]{Brezis2011}
\begin{equation}
\label{eq:varphi}
\|\widehat{e_{ij}} - \varphi_{ij}\|_{L^1(\mathbb{R}^d)}
< \frac{\varepsilon}{2C_1Ld^2}.
\end{equation}
Moreover, there exists $ 0 < \delta_2 < R $ such that, for every $ 1 \le i,j \le d $,
\[
|\varphi_{ij}(\xi) - \varphi_{ij}(\xi')|
< \frac{\varepsilon}{C_1|B(0,2R)|Ld^2},
\quad
\text{for all }  \xi,  \xi' \in  \mathbb{R}^d
\text{ with } | \xi - \xi' | < \delta_2.
\]
Here $ |B(0,2R)| $ denotes the Lebesgue measure of the ball in $ \mathbb{R}^d $ centered at the origin with radius $ 2R $.
\item For every $ 1 \le i,j \le d $, since $ \widehat{\partial_i e_{ij}} \in L^1(\mathbb{R}^d) $ by Proposition~\ref{prop:mainprop},  
there exists $ \psi_{ij} \in C_c^{\infty}(\mathbb{R}^d) $ such that
\begin{equation}
\label{eq:psi}
\|\widehat{\partial_i e_{ij}} - \psi_{ij}\|_{L^1(\mathbb{R}^d)}
< \frac{\varepsilon}{2C_1Ld^2}.
\end{equation}
Moreover, there exists $ 0 < \delta_3 < R $ such that, for every $ 1 \le i,j \le d $,
\[
|\psi_{ij}(\xi) - \psi_{ij}(\xi')|
< \frac{\varepsilon}{C_1|B(0,2R)|Ld^2},
\quad
\text{for all } \xi,  \xi' \in  \mathbb{R}^d
\text{ with } | \xi - \xi' | < \delta_3.
\]

\item Since $ \widehat{w} \in L^1(\mathbb{R}^d) $,  
there exists $ \phi \in C_c^{\infty}(\mathbb{R}^d) $ such that
\begin{equation}
\label{eq:phi}
\|\widehat{w} - \phi\|_{L^1(\mathbb{R}^d)} 
< \frac{\varepsilon}{2C_1}.
\end{equation}
Moreover, there exists $ 0 < \delta_4 < R $ such that
\[
|\phi(\xi) - \phi(\xi')|
< \frac{\varepsilon}{C_1|B(0,2R)|},
\quad
\text{for all } \xi,  \xi' \in  \mathbb{R}^d
\text{ with } | \xi - \xi' | < \delta_4.
\]

\item For every $ 1 \le i\le d $, since $ \widehat{v_i} \in L^1(\mathbb{R}^d) $,  
there exists $ \sigma_{i} \in C_c^{\infty}(\mathbb{R}^d) $ such that
\begin{equation}
\label{eq:sigma}
\|\widehat{v_{i}} - \sigma_{i}\|_{L^1(\mathbb{R}^d)}
< \frac{\varepsilon}{2C_1d}.
\end{equation}
Moreover, there exists $ 0 < \delta_5 < R $ such that, for every $ 1 \le i \le d $,
\[
|\sigma_{i}(\xi) - \sigma_{i}(\xi')|
< \frac{\varepsilon}{C_1|B(0,2R)|d},
\quad
\text{for all } \xi,  \xi' \in  \mathbb{R}^d
\text{ with } | \xi - \xi' | < \delta_5.
\]
\end{itemize}

We now begin the verification of condition~(3).  
For every $ u \in \mathcal{B}^{s+1} $ with $ \|u\|_{\mathcal{B}^{s+1}} \le 1 $,  
if $ |y| < \min\{\delta_1, \delta_2, \delta_3, \delta_4, \delta_5\}<R $,  
we compute
\begin{align}
&  \int_{\mathbb{R}^d} \left| \widehat{\mathcal{T}(u)}(\xi + y) (1 + | \xi + y |^2)^{(s+1)/2} 
- \widehat{\mathcal{T}(u)}(\xi) (1 + | \xi |^2)^{(s+1)/2} \right|\, \mathrm{d}\xi \notag \\
\leq& 
\int_{| \xi | > 2R} \left|\widehat{F}(\xi + y) (1 + | \xi + y |^2)^{(s+1)/2}\right|\, \mathrm{d}\xi 
+ \int_{| \xi | > 2R} \left|\widehat{F}(\xi) (1 + | \xi |^2)^{(s+1)/2}\right|\, \mathrm{d}\xi  \notag\\
& \quad +  \int_{| \xi | \leq 2R} \left| \widehat{F}(\xi + y) (1 + | \xi + y |^2)^{(s+1)/2} - \widehat{F}(\xi) (1 + | \xi |^2)^{(s+1)/2}   \right|\, \mathrm{d}\xi \notag\\
\leq& 
2 \int_{| \xi | > R} \left| \widehat{F}(\xi) (1 + | \xi |^2)^{(s+1)/2}\right| \, \mathrm{d}\xi  \notag\\
&\quad + \int_{| \xi | \leq 2R} \left| \widehat{F}(\xi + y) (1 + | \xi + y |^2)^{(s+1)/2} - \widehat{F}(\xi) (1 + | \xi |^2)^{(s+1)/2} \right|\, \mathrm{d}\xi \notag\\
< &
2\varepsilon  + \int_{| \xi | \leq 2R} \left| \widehat{F}(\xi + y) (1 + | \xi + y |^2)^{(s+1)/2} - \widehat{F}(\xi) (1 + | \xi |^2)^{(s+1)/2} \right|\, \mathrm{d}\xi. \label{eq:last}
\end{align}
To estimate the last term in~\eqref{eq:last}, we decompose it as
\begin{align}
&\int_{| \xi | \leq 2R} \left| \widehat{F}(\xi + y) (1 + | \xi + y |^2)^{(s+1)/2} - \widehat{F}(\xi) (1 + | \xi |^2)^{(s+1)/2} \right|\, \mathrm{d}\xi \notag \\
\leq & \int_{| \xi | \leq 2R} \left| \widehat{F-G}(\xi + y) (1 + | \xi + y |^2)^{(s+1)/2} - \widehat{F-G}(\xi) (1 + | \xi |^2)^{(s+1)/2} \right|\, \mathrm{d}\xi \tag{E1}\label{eq:E1} \\
&\quad + \int_{| \xi | \leq 2R} \left| \widehat{G}(\xi + y) (1 + | \xi + y |^2)^{(s+1)/2} - \widehat{G}(\xi) (1 + | \xi |^2)^{(s+1)/2} \right|\, \mathrm{d}\xi. \tag{E2}\label{eq:E2}
\end{align}

\noindent\textbf{Estimation of \eqref{eq:E1}.} 
By~\eqref{eq:FGH}, we have
\begin{align}
	&\int_{| \xi | \leq 2R} \left| \widehat{F-G}(\xi + y) (1 + | \xi + y |^2)^{(s+1)/2} - \widehat{F-G}(\xi) (1 + | \xi |^2)^{(s+1)/2} \right|\, \mathrm{d}\xi \notag \\
	=& \int_{| \xi | \leq 2R} \left| \frac{(1 + | \xi + y |^2)^{(s+1)/2}}{\alpha+i(\xi+y)^{\top}\beta+(\xi+y)^{\top}M(\xi+y)} \widehat{H}(\xi + y) -  \frac{(1 + | \xi |^2)^{(s+1)/2}}{\alpha+i\xi^{\top}\beta+\xi^{\top}M\xi} \widehat{H}(\xi) \right|\, \mathrm{d}\xi \notag \\ 
	\leq& \int_{| \xi | \leq 2R} \left| \frac{(1 + | \xi + y |^2)^{(s+1)/2}}{\alpha+i(\xi+y)^{\top}\beta+(\xi+y)^{\top}M(\xi+y)}  -  \frac{(1 + | \xi |^2)^{(s+1)/2}}{\alpha+i\xi^{\top}\beta+\xi^{\top}M\xi}  \right|\cdot  \left|\widehat{H}(\xi + y)\right|\, \mathrm{d}\xi \notag \\ 
	&\quad + \int_{| \xi | \leq 2R} \left|  \frac{(1 + | \xi |^2)^{(s+1)/2}}{\alpha+i\xi^{\top}\beta+\xi^{\top}M\xi}  \right| \cdot \left|\widehat{H}(\xi + y)-\widehat{H}(\xi )\right|\, \mathrm{d}\xi \notag \\
	\leq& \underbrace{\varepsilon  \int_{| \xi | \leq 3R}  \left|\widehat{H}(\xi )\right|\, \mathrm{d}\xi }_{(\text{E1-A})} + \underbrace{C_1  \int_{| \xi | \leq 2R}  \left|\widehat{H}(\xi + y)-\widehat{H}(\xi )\right|\, \mathrm{d}\xi}_{(\text{E1-B})}. \label{eq:34}
\end{align}

\noindent\textbf{Estimation of (\text{E1-A}).} 
By~\eqref{eq:FGH} and Lemma~\ref{lem:mainlem}, we compute
\begin{align}
&\varepsilon  \int_{| \xi | \leq 3R}  \left|\widehat{H}(\xi )\right|\, \mathrm{d}\xi =\varepsilon  \int_{| \xi | \leq 3R} \left| \alpha+i\xi^{\top}\beta+\xi^{\top}M\xi\right|\cdot  \left|\widehat{F-G}(\xi )\right|\, \mathrm{d}\xi 	\notag \\
\leq & \varepsilon \int_{| \xi | \leq 3R} \left| \alpha+i\xi^{\top}\beta+\xi^{\top}M\xi\right|\cdot \left|\widehat{F}(\xi )\right|\, \mathrm{d}\xi + \varepsilon  \int_{| \xi | \leq 3R} \left| \alpha+i\xi^{\top}\beta+\xi^{\top}M\xi\right|\cdot \left|\widehat{G}(\xi )\right|\, \mathrm{d}\xi  	\notag \\
=&  \varepsilon \int_{| \xi | \leq 3R}L^{-1} \left|L(\alpha+i\xi^{\top}\beta+ \xi^{\top}M\xi)\right| \cdot  \left|\widehat{F}(\xi )\right|\, \mathrm{d}\xi + \varepsilon \int_{| \xi | \leq 3R} \left|\widehat{wu}(\xi )+ \widehat{v\cdot \nabla u}(\xi) \right|\, \mathrm{d}\xi  	\notag \\
\leq& \varepsilon C_2 \int_{\BR^d}L^{-1} \left|\widehat{F}(\xi )\right|\, \mathrm{d}\xi + \varepsilon \int_{\BR^d} \left|\widehat{wu}(\xi )+\widehat{v\cdot \nabla u}(\xi)\right|\, \mathrm{d}\xi  	\notag \\
\leq& \varepsilon C_2 \|wu+v\cdot \nabla u \|_{\CB^s}+\varepsilon \|wu+v\cdot \nabla u\|_{\CB^0}
\leq (\varepsilon C_2+\varepsilon) (\|w\|_{\CB^s}+\|v\|_{\CB^s}). \label{eq:estE1-1}
\end{align}

\noindent\textbf{Estimation of (\text{E1-B}).} 
We compute
\begin{align}
	&C_1 \int_{| \xi | \leq 2R}  \left|\widehat{H}(\xi + y)-\widehat{H}(\xi )\right|\, \mathrm{d}\xi \notag \\
	\leq& \sum_{i,j=1}^{d}C_1  \int_{| \xi | \leq 2R} \left| \widehat{e_{ij}\partial_{ij}F}(\xi+y)-\widehat{ e_{ij}\partial_{ij}F}(\xi)\right|\, \mathrm{d}\xi \notag \\
	&\quad  + \sum_{i,j=1}^{d}C_1 \int_{| \xi | \leq 2R} \left| \widehat{\partial_i e_{ij}\partial_{j}F}(\xi+y)-\widehat{\partial_i e_{ij}\partial_{j}F}(\xi)\right|\, \mathrm{d}\xi \notag \\
	\leq& \sum_{i,j=1}^{d}C_1 \int_{| \xi | \leq 2R}\int_{\BR^d} \left| \widehat{\partial_{ij}F}(\eta)\right|\cdot  \left| \widehat{e_{ij}}(\xi+y-\eta)-\widehat{e_{ij}}(\xi-\eta)\right|\, \mathrm{d}\eta\mathrm{d}\xi \notag \\
	&\quad + \sum_{i,j=1}^{d}C_1 \int_{| \xi | \leq 2R}\int_{\BR^d} \left| \widehat{\partial_{j}F}(\eta)\right|\cdot  \left| \widehat{\partial_i e_{ij}}(\xi+y-\eta)-\widehat{\partial_i e_{ij}}(\xi-\eta)\right|\, \mathrm{d}\eta\mathrm{d}\xi \notag \\
	\leq& \sum_{i,j=1}^{d}C_1 \int_{| \xi | \leq 2R}\int_{\BR^d} \left| \widehat{\partial_{ij}F}(\eta)\right|\cdot \left| \widehat{e_{ij}}(\xi+y-\eta)- \varphi_{ij}(\xi+y-\eta)\right|\, \mathrm{d}\eta\mathrm{d}\xi \notag \\
	&\quad + \sum_{i,j=1}^{d} C_1 \int_{| \xi | \leq 2R}\int_{\BR^d} \left| \widehat{\partial_{ij}F}(\eta)\right|\cdot \left| \widehat{e_{ij}}(\xi-\eta)- \varphi_{ij}(\xi-\eta)\right|\, \mathrm{d}\eta\mathrm{d}\xi \notag \\
	&\quad + \sum_{i,j=1}^{d} C_1 \int_{| \xi | \leq 2R}\int_{\BR^d} \left| \widehat{\partial_{ij}F}(\eta)\right|\cdot \left| \varphi_{ij}(\xi+y-\eta)- \varphi_{ij}(\xi-\eta)\right|\, \mathrm{d}\eta\mathrm{d}\xi \notag \\
	&\quad + \sum_{i,j=1}^{d}C_1 \int_{| \xi | \leq 2R}\int_{\BR^d} \left| \widehat{\partial_{j}F}(\eta)\right|\cdot \left| \widehat{\partial_i e_{ij}}(\xi+y-\eta)- \psi_{ij}(\xi+y-\eta)\right|\, \mathrm{d}\eta\mathrm{d}\xi \notag \\
	&\quad + \sum_{i,j=1}^{d} C_1 \int_{| \xi | \leq 2R}\int_{\BR^d} \left| \widehat{\partial_{j}F}(\eta)\right|\cdot \left| \widehat{\partial_i e_{ij}}(\xi-\eta)- \psi_{ij}(\xi-\eta)\right|\, \mathrm{d}\eta\mathrm{d}\xi \notag \\
	&\quad + \sum_{i,j=1}^{d} C_1 \int_{| \xi | \leq 2R}\int_{\BR^d} \left| \widehat{\partial_{j}F}(\eta)\right|\cdot \left| \psi_{ij}(\xi+y-\eta)- \psi_{ij}(\xi-\eta)\right|\, \mathrm{d}\eta\mathrm{d}\xi \notag \\
	\leq& \sum_{i,j=1}^{d}2C_1 \int_{\BR^d}\left| \widehat{\partial_{ij}F}(\eta)\right| \|\widehat{e_{ij}}-\varphi_{ij}\|_{L^1(\BR^d)}\, \mathrm{d}\eta \notag \\
	&\quad + \sum_{i,j=1}^{d}\frac{\varepsilon}{|B(0,2R)|Ld^2}  \int_{| \xi | \leq 2R}\int_{\BR^d} \left| \widehat{\partial_{ij}F}(\eta)\right|\, \mathrm{d}\eta\mathrm{d}\xi  \notag \\
	&\quad + \sum_{i,j=1}^{d}2C_1 \int_{\BR^d}\left| \widehat{\partial_{j}F}(\eta)\right| \|\widehat{\partial_i e_{ij}}-\psi_{ij}\|_{L^1(\BR^d)}\, \mathrm{d}\eta  \notag \\
	&\quad+ \sum_{i,j=1}^{d}\frac{\varepsilon}{|B(0,2R)|Ld^2}  \int_{| \xi | \leq 2R}\int_{\BR^d} \left| \widehat{\partial_{j}F}(\eta)\right|\, \mathrm{d}\eta\mathrm{d}\xi  \notag \\
	\leq& \sum_{i,j=1}^{d} \frac{2\varepsilon}{Ld^2} \int_{\BR^d}\left| \widehat{\partial_{ij}F}(\eta)\right|\, \mathrm{d}\eta
	+\sum_{i,j=1}^{d} \frac{2\varepsilon}{Ld^2} \int_{\BR^d}\left| \widehat{\partial_{j}F}(\eta)\right|\, \mathrm{d}\eta \notag \\
	\leq & \sum_{i,j=1}^{d} \frac{2 \varepsilon}{Ld^2} \|\partial_{ij}F\|_{\CB^s}
	+\sum_{i,j=1}^{d} \frac{2\varepsilon}{Ld^2} \|\partial_{j}F\|_{\CB^s}  \notag \\
	\leq& 4\varepsilon \|wu+v\cdot \nabla u \|_{\CB^s} \leq 4\varepsilon (\|w\|_{\CB^s}+\|v\|_{\CB^s})
	\label{eq:estE1-2}
\end{align}
Combining~\eqref{eq:estE1-1} and~\eqref{eq:estE1-2}, we obtain
\begin{equation}\label{eq:resultE1}
\begin{split}
	&\int_{| \xi | \leq 2R} \left| \widehat{F-G}(\xi + y) (1 + | \xi + y |^2)^{(s+1)/2} - \widehat{F-G}(\xi) (1 + | \xi |^2)^{(s+1)/2} \right|\, \mathrm{d}\xi \\
	\leq& ( 5\varepsilon+C_2\varepsilon) (\|w\|_{\CB^s}+\|v\|_{\CB^s}).
\end{split}
\end{equation}

\noindent\textbf{Estimation of \eqref{eq:E2}.}
By~\eqref{eq:FGH}, we obtain
\begin{align}
	 &\int_{| \xi | \leq 2R} \left| \widehat{G}(\xi + y) (1 + | \xi + y |^2)^{(s+1)/2} - \widehat{G}(\xi) (1 + | \xi |^2)^{(s+1)/2} \right|\, \mathrm{d}\xi \notag \\
	 = &  \int_{| \xi | \leq 2R} \left| \frac{(1 + | \xi + y |^2)^{(s+1)/2}}{\alpha+i(\xi+y)^{\top}\beta+ (\xi+y)^{\top}M(\xi+y)}(\widehat{wu}(\xi+y)+\widehat{v\cdot \nabla u}(\xi+y))\right.  \notag \\ 
	 &\qquad\qquad\qquad\qquad \left. -\frac{(1 + | \xi |^2)^{(s+1)/2}}{\alpha+i\xi^{\top}\beta+\xi^{\top}M\xi}(\widehat{wu}(\xi)+\widehat{v\cdot \nabla u}(\xi))\right|  \, \mathrm{d}\xi \notag \\  
	 \leq& \int_{| \xi | \leq 2R} \left| \frac{(1 + | \xi + y |^2)^{(s+1)/2}}{\alpha+i(\xi+y)^{\top}\beta+(\xi+y)^{\top}M(\xi+y)}  -  \frac{(1 + | \xi |^2)^{(s+1)/2}}{\alpha+i\xi^{\top}\beta+\xi^{\top}M\xi}  \right| \notag \\ 
	 &\qquad\qquad\qquad\qquad \cdot \left|\widehat{wu}(\xi + y)+\widehat{v\cdot \nabla u}(\xi+y)\right|\, \mathrm{d}\xi \notag \\ 
	&\quad + \int_{| \xi | \leq 2R} \left|  \frac{(1 + | \xi |^2)^{(s+1)/2}}{\alpha+i\xi^{\top}\beta+\xi^{\top}M\xi}  \right|\cdot \left|(\widehat{wu}(\xi + y)+\widehat{v\cdot \nabla u}(\xi+y))-(\widehat{wu}(\xi )+\widehat{v\cdot \nabla u}(\xi))\right|\, \mathrm{d}\xi \notag \\
	\leq& \underbrace{\varepsilon \int_{| \xi | \leq 3R}  \left|\widehat{wu}(\xi )+\widehat{v\cdot \nabla u}(\xi)\right|\, \mathrm{d}\xi }_{(\text{E2-A})}  \notag \\
	&\quad + \underbrace{C_1 \int_{| \xi | \leq 2R}  \left|(\widehat{wu}(\xi + y)+\widehat{v\cdot \nabla u}(\xi+y))-(\widehat{wu}(\xi )+\widehat{v\cdot \nabla u}(\xi))\right|\, \mathrm{d}\xi}_{(\text{E2-B})}. 
\end{align}
\noindent\textbf{$\bullet$ Estimation of (\text{E2-A}).} 
This follows immediately from the definition of the $ \mathcal{B}^s $-norm together with Proposition~\ref{prop:mainprop}:
\begin{equation}\label{eq:estE2-A1}
\varepsilon \int_{| \xi | \leq 3R}  \left|\widehat{wu}(\xi )+\widehat{v\cdot \nabla u}(\xi)\right|\, \mathrm{d}\xi
\leq  \varepsilon\|wu+v\cdot \nabla u\|_{\CB^s}
\leq \varepsilon (\|w\|_{\CB^s}+\|v\|_{\CB^s}).
\end{equation}

\noindent\textbf{$\bullet$ Estimation of (\text{E2-B}).} 
We compute
\begin{align}
&C_1 \int_{| \xi | \leq 2R}  \left|(\widehat{wu}(\xi + y)+\widehat{v\cdot \nabla u}(\xi+y))-(\widehat{wu}(\xi )+\widehat{v\cdot \nabla u}(\xi))\right|\, \mathrm{d}\xi   \notag \\
\leq &C_1 \int_{| \xi | \leq 2R} \int_{\BR^d}\left|\widehat{u}(\eta)\right|\cdot \left| \widehat{w}(\xi+y-\eta)- \widehat{w}(\xi-\eta)\right| \, \mathrm{d}\eta\mathrm{d}\xi  \notag \\
&\quad  + C_1   \sum_{i=1}^{d}\int_{| \xi | \leq 2R} \int_{\BR^d}\left|\widehat{\partial_i u}(\eta)\right|\cdot \left| \widehat{v_i}(\xi+y-\eta)- \widehat{v_i}(\xi-\eta)\right| \, \mathrm{d}\eta\mathrm{d}\xi \notag \\
\leq& C_1 \int_{| \xi | \leq 2R}\int_{\BR^d} \left| \widehat{u}(\eta)\right|\cdot \left| \widehat{w}(\xi+y-\eta)- \phi(\xi+y-\eta)\right|\, \mathrm{d}\eta\mathrm{d}\xi \notag \\
	&\quad +  C_1 \int_{| \xi | \leq 2R}\int_{\BR^d} \left| \widehat{u}(\eta)\right|\cdot \left| \widehat{w}(\xi-\eta)- \phi(\xi-\eta)\right|\, \mathrm{d}\eta\mathrm{d}\xi \notag \\
	&\quad + C_1 \int_{| \xi | \leq 2R}\int_{\BR^d} \left| \widehat{u}(\eta)\right|\cdot \left| \phi(\xi+y-\eta)- \phi(\xi-\eta)\right|\, \mathrm{d}\eta\mathrm{d}\xi \notag \\
	&\quad + C_1\sum_{i=1}^{d} \int_{| \xi | \leq 2R}\int_{\BR^d} \left| \widehat{\partial_i u}(\eta)\right|\cdot \left| \widehat{v_i}(\xi+y-\eta)- \sigma_i(\xi+y-\eta)\right|\, \mathrm{d}\eta\mathrm{d}\xi \notag \\
	&\quad +  C_1\sum_{i=1}^{d} \int_{| \xi | \leq 2R}\int_{\BR^d} \left| \widehat{\partial_i u}(\eta)\right|\cdot \left| \widehat{v_i}(\xi-\eta)- \sigma_i(\xi-\eta)\right|\, \mathrm{d}\eta\mathrm{d}\xi \notag \\
	&\quad + C_1\sum_{i=1}^{d} \int_{| \xi | \leq 2R}\int_{\BR^d} \left| \widehat{\partial_i u}(\eta)\right|\cdot \left| \sigma_i(\xi+y-\eta)- \sigma_i(\xi-\eta)\right|\, \mathrm{d}\eta\mathrm{d}\xi \notag \\
	\leq& 2C_1 \int_{\BR^d}\left| \widehat{u}(\eta)\right| \|\widehat{w}-\phi\|_{L^1(\BR^d)}\, \mathrm{d}\eta + \frac{\varepsilon}{|B(0,2R)|}  \int_{| \xi | \leq 2R}\int_{\BR^d} \left| \widehat{u}(\eta)\right|\, \mathrm{d}\eta\mathrm{d}\xi  \notag \\
	&\quad + 2C_1\sum_{i=1}^{d} \int_{\BR^d}\left| \widehat{\partial_i u}(\eta)\right| \|\widehat{v_i}-\sigma_i\|_{L^1(\BR^d)}\, \mathrm{d}\eta + \frac{\varepsilon}{|B(0,2R)|d}\sum_{i=1}^{d}  \int_{| \xi | \leq 2R}\int_{\BR^d} \left| \widehat{\partial_i u}(\eta)\right|\, \mathrm{d}\eta\mathrm{d}\xi  \notag \\
	\leq& 2\varepsilon \int_{\BR^d}\left| \widehat{u}(\eta)\right|\, \mathrm{d}\eta 
	+\frac{2\varepsilon}{d}\sum_{i=1}^{d}\int_{\BR^d} \left| \widehat{\partial_i u}(\eta)\right|\, \mathrm{d}\eta \notag \\
	\leq& 2\varepsilon \|u\|_{\CB^s}+2\varepsilon \|u\|_{\CB^{s+1}}\leq 4\varepsilon. \label{eq:estE2-A2}
\end{align}
Combining~\eqref{eq:estE2-A1} and~\eqref{eq:estE2-A2}, we obtain
\begin{equation}\label{eq:resultE2}
\begin{split}
	&\int_{| \xi | \leq 2R} \left| \widehat{G}(\xi + y) (1 + | \xi + y |^2)^{(s+1)/2} - \widehat{G}(\xi) (1 + | \xi |^2)^{(s+1)/2} \right|\, \mathrm{d}\xi \\
	\leq& (\|w\|_{\CB^s}+\|v\|_{\CB^s}+4)
	\varepsilon.
\end{split}
\end{equation}

Finally, combining~\eqref{eq:last}, \eqref{eq:E1}, \eqref{eq:E2}, \eqref{eq:resultE1}, and~\eqref{eq:resultE2}, we conclude that for every 
$ u \in \mathcal{B}^{s+1} $ with $ \|u\|_{\mathcal{B}^{s+1}} \le 1 $,
if $ |y| < \min\{\delta_1, \delta_2, \delta_3, \delta_4, \delta_5\} $, then
\begin{equation*}
	\begin{split}
&\int_{\mathbb{R}^d}
\left|
\widehat{\mathcal{T}(u)}(\xi + y) (1 + | \xi + y |^2)^{(s+1)/2}
- \widehat{\mathcal{T}(u)}(\xi) (1 + | \xi |^2)^{(s+1)/2}
\right|
\, \mathrm{d}\xi \\
\le & \left(6+(6+C_2)(\|w\|_{\mathcal{B}^s} + \|v\|_{\mathcal{B}^s})\right) \varepsilon.
\end{split}
\end{equation*}
This completes the verification.
\end{proof}

\begin{lemma}\label{lem:bdd}
Suppose that (A1)--(A3) in Assumption~\ref{assum} hold with $s\geq 0$.  
Then the operator
\[
(I + \mathcal{T})^{-1} \colon \mathcal{B}^{s+1} \longrightarrow \mathcal{B}^{s+1}
\]
is bounded.
\end{lemma}

\begin{proof}
Since $ \mathcal{T}\colon \mathcal{B}^{s+1}\rightarrow \mathcal{B}^{s+1} $ has been proved to be compact in Lemma~\ref{lem:compactT},  
the operator $ I + \mathcal{T} $ is a Fredholm operator with index zero.  
Therefore, to show that $ I + \mathcal{T} $ admits a bounded inverse,  
it suffices to verify that $ I + \mathcal{T} $ is injective.

Suppose for the sake of contradiction that there exists a nonzero function  
$ u \in \mathcal{B}^{s+1} $ satisfying
\begin{equation}
\label{ITu}
(I + \mathcal{T})u = 0.
\end{equation}

We first claim that $ u \in \mathcal{B}^{s+2} $.  
Indeed, equation~\eqref{ITu} is equivalent to
\[
(\alpha +\beta\cdot \nabla - \nabla \cdot A(x) \nabla)u = -w u - v \cdot \nabla u.
\]
The right-hand side lies in $ \mathcal{B}^s $ by Proposition~\ref{prop:mainprop};  
by Lemma~\ref{lem:mainlem} it follows that $ u \in \mathcal{B}^{s+2} $.

Assume $ u(x_0) \neq 0 $ for some $ x_0 \in \mathbb{R}^d $.  
Since $ \widehat{u} \in L^1(\mathbb{R}^d) $, the Riemann--Lebesgue lemma gives
\[
\lim_{|x| \to \infty} |u(x)| = 0.
\]
Hence there exists $ R > |x_0| $ such that 
\begin{equation}
\label{eq:uR}
|u(x)| < |u(x_0)| 
\quad \text{for all } x \in \mathbb{R}^d \text{ with } |x| \ge R.
\end{equation}

By Proposition \ref{prop:mainprop}, $u\in \CB^{s+2}$ implies $ u \in C^2(\mathbb{R}^d) $.  
From~\eqref{ITu} it also follows that $ u $ satisfies
\[
- \nabla \cdot (A(x) \nabla u) + b(x) \cdot \nabla u + c(x) u = 0.
\]
Applying the weak maximum principle~\cite[Theorem~6.4.2]{Evans10} to $ u $ on $ B(0,R) $ yields
\[
\sup_{|x| \le R} |u(x)|
= \sup_{|x| = R} |u(x)|.
\]
Therefore, there exists a point $ x' \in \mathbb{R}^d $ with $ |x'| = R $ such that 
\[
|u(x')| \ge |u(x_0)|.
\]
This contradicts \eqref{eq:uR}.
\end{proof}

\begin{proof}[\textbf{Proof of Theorem~\ref{thm:regularitythm}}]
The existence and uniqueness of $ u^* \in \mathcal{B}^{s+1} $ follow from the injectivity of 
$ (\alpha+\beta\cdot \nabla - \nabla \cdot A(x) \nabla)^{-1} $ and $ (I + \mathcal{T})^{-1} $,  
as shown in Lemma~\ref{lem:mainlem} and in the proof of Lemma~\ref{lem:bdd} respectively, together with
\[
u^*
= (I + \mathcal{T})^{-1}
(\alpha+\beta\cdot \nabla - \nabla \cdot A(x) \nabla)^{-1} f.
\]
Applying the same regularity bootstrap argument used in the proof of Lemma~\ref{lem:bdd},  
we conclude that $ u^* \in \mathcal{B}^{s+2} $.

Applying Proposition~\ref{prop:mainprop} and Lemma~\ref{lem:mainlem} to 
\[
(\alpha+\beta\cdot \nabla - \nabla \cdot A(x) \nabla)u^* = f - w u^*-v\cdot \nabla u^*,
\]
we obtain
\begin{equation}\label{u2}
\norm{u^*}_{\mathcal{B}^{s+2}}
\le L\norm{f - wu^*-v\cdot \nabla u^*}_{\mathcal{B}^s} 
\le L
    \left( \norm{f}_{\mathcal{B}^s} + (\norm{w}_{\mathcal{B}^s}+ \|v\|_{\CB^{s}})\norm{u^*}_{\mathcal{B}^{s+1}} \right). 
\end{equation}
By Lemma~\ref{lem:bdd} and Lemma~\ref{lem:mainlem}, we have 
\begin{equation}\label{u1}
\begin{split}
\|u^*\|_{\mathcal{B}^{s+1}}
=& \left\| (I + \mathcal{T})^{-1} 
 (\alpha+\beta\cdot \nabla - \nabla \cdot A(x) \nabla)^{-1} f \right\|_{\mathcal{B}^{s+1}} \\
\le & L
\left\| (I + \mathcal{T})^{-1} \right\|_{\mathcal{B}^{s+1} \to \mathcal{B}^{s+1}} 
\|f\|_{\mathcal{B}^s}.
\end{split} 
\end{equation}
Combining~\eqref{u2} and~\eqref{u1}, we obtain
		\begin{equation*}
			\norm{u^*}_{\CB^{s+2}}\leq L \left(1+ L(\norm{w}_{\CB^s}+\norm{v}_{\CB^s})\cdot \norm{(I+\CT)^{-1}}_{\CB^{s+1}\rightarrow\CB^{s+1}}\right)\norm{f}_{\CB^s},
		\end{equation*}
where $L$ and $\CT$ are given by \eqref{eq:defL} and \eqref{eq:opT}.   
Let
\begin{equation}
\label{eq:elliptic-bound}
C\coloneqq L \left(1+ L(\norm{w}_{\CB^s}+\norm{v}_{\CB^s})
\cdot \norm{(I+\CT)^{-1}}_{\CB^{s+1}\rightarrow\CB^{s+1}}\right).
\end{equation}
This completes the proof.
\end{proof}

\begin{proof}[\textbf{Proof of Corollary~\ref{cor:dimind}}]
For any \( u \in \CB^{s+1} \), by Lemma~\ref{lem:mainlem} and
Proposition~\ref{prop:mainprop}, we obtain
\begin{equation}
\|\CT(u)\|_{\CB^{s+1}}
\le 2^{-1/2}\|\CT(u)\|_{\CB^{s+2}}
\le 2^{-1/2}L\|wu+v\cdot \nabla u\|_{\CB^s}
\le
\frac{\|w\|_{\CB^s}+\|v\|_{\CB^s}}
{\min\{\alpha,m\}-\|E\|_{\CB^{s+1}}}
\|u\|_{\CB^{s+1}}.
\end{equation}
Consequently, under (A3') in Assumption~\ref{assum}, it follows that
\[
\|\CT\|_{\CB^{s+1}\to \CB^{s+1}}
\le
\frac{\|w\|_{\CB^s}+\|v\|_{\CB^s}}
{\min\{\alpha,m\}-\|E\|_{\CB^{s+1}}}
<1.
\]
A Neumann series argument yields
\begin{equation}
\label{eq:newIT}
\|(I+\CT)^{-1}\|_{\CB^{s+1}\to \CB^{s+1}}
\le \sum_{i=0}^{\infty}\|\CT\|_{\CB^{s+1}\to \CB^{s+1}}^{\,i}
\le
\frac{\min\{\alpha,m\}-\|E\|_{\CB^{s+1}}}
{\min\{\alpha,m\}-\|E\|_{\CB^{s+1}}-\|w\|_{\CB^s}-\|v\|_{\CB^s}}.
\end{equation}
Combining \eqref{eq:elliptic-bound} and \eqref{eq:newIT}, we obtain
\begin{equation}
    \label{eq:newC}
    \begin{split}
        C &\le
\frac{2(\min\{\alpha,m\}-\|E\|_{\CB^{s+1}}+\|w\|_{\CB^s}+\|v\|_{\CB^s})}
{(\min\{\alpha,m\}-\|E\|_{\CB^{s+1}})(\min\{\alpha,m\}-\|E\|_{\CB^{s+1}}-\|w\|_{\CB^s}-\|v\|_{\CB^s})} \\
&\leq \frac{2(\min\{\alpha,m\}+K)}
{\min\{\alpha,m\}(\min\{\alpha,m\}-K)}. 
    \end{split}
\end{equation}
The corollary then follows from Theorem~\ref{thm:regularitythm} and~\eqref{eq:newC}.
\end{proof}

\appendix

\section{An Alternative Proof of Lemma~\ref{lem:mainlem}}
\label{app}

In this appendix, we present an alternative proof of Lemma~\ref{lem:mainlem} using methods from perturbation theory. A similar idea was used in the proof of \cite[Theorem~3.8]{shuailu25} in a special case of our setting. Let $\CL$ denote the elliptic operator in \eqref{eq:rewrite}, given by
\[
\CL \coloneqq -\nabla \cdot (A(x)\nabla) + \beta\cdot\nabla+\alpha.
\]
We introduce the constant-coefficient operator
\[
\CL_0 \coloneqq -\nabla\cdot (M\nabla)+\beta\cdot\nabla+\alpha.
\]
A standard Fourier multiplier argument yields the coercivity estimate
\[
\|\CL_0 u\|_{\CB^s} \ge \frac{\min\{\alpha,m\}}{2}\,\|u\|_{\CB^{s+2}},
\quad \text{for all } u\in \CB^{s+2},
\]
and consequently
\begin{equation}\label{eq:L0inv}
\|\CL_0^{-1} g\|_{\CB^{s+2}}
\le \frac{2}{\min\{\alpha,m\}}\,\|g\|_{\CB^s},
\quad \text{for all } g\in \CB^s .
\end{equation}

Since $A(x)=M+E(x)$, we may write
\[
\CL = \CL_0 - \nabla\cdot(E(x)\nabla)
      = \CL_0(\Id-\CK),
\qquad
\CK \coloneqq \CL_0^{-1}\,\nabla\cdot\bigl(E(x)\nabla\bigr).
\]
Therefore, the equation $\CL u=f$ is equivalent to
\begin{equation}\label{eq:IdK}
(\Id-\CK)u = \CL_0^{-1}f .
\end{equation}

We now estimate the operator $\CK$ on $\CB^{s+2}$.  
Using the divergence form
\[
\nabla\cdot (E(x)\nabla u)
= \sum_{i=1}^d \partial_i\left(\sum_{j=1}^d e_{ij}\partial_j u\right),
\]
together with the two basic Barron space estimates from Proposition~\ref{prop:mainprop},
we obtain, for each $i,j$,
\[
\|\partial_i(e_{ij}\partial_j u)\|_{\CB^s}
\le 2^{-1/2}\|e_{ij}\partial_j u\|_{\CB^{s+1}}
\le 2^{-1/2}\|e_{ij}\|_{\CB^{s+1}}\|\partial_j u\|_{\CB^{s+1}}
\le \frac{1}{2}\|e_{ij}\|_{\CB^{s+1}}\|u\|_{\CB^{s+2}}.
\]
Summing over $i$ and $j$ yields
\begin{equation}\label{eq:divEgrad}
\|\nabla\cdot (E(x)\nabla u)\|_{\CB^s}
\le \frac{1}{2}\|E\|_{\CB^{s+1}}\|u\|_{\CB^{s+2}}.
\end{equation}
Combining \eqref{eq:L0inv} and \eqref{eq:divEgrad}, we obtain
\[
\|\CK u\|_{\CB^{s+2}}
\le \frac{2}{\min\{\alpha,m\}}\cdot \frac{1}{2}\|E\|_{\CB^{s+1}}\|u\|_{\CB^{s+2}}
= \frac{\|E\|_{\CB^{s+1}}}{\min\{\alpha,m\}}\,
\|u\|_{\CB^{s+2}}.
\]
Hence, in order for the operator $\Id-\CK$ to be invertible on the Banach space $B^{s+2}$, it suffices that
\[
\|\CK\|_{\CB^{s+2}\to \CB^{s+2}}
\le \frac{\|E\|_{\CB^{s+1}}}{\min\{\alpha,m\}} < 1 .
\]
Under this condition, a Neumann series argument yields
\[
\|(\Id-\CK)^{-1}\|
\le \frac{1}{1-\|\CK\|}
\le \frac{1}{1-\frac{\|E\|_{\CB^{s+1}}}{\min\{\alpha,m\}}}
= \frac{\min\{\alpha,m\}}{\min\{\alpha,m\}-\|E\|_{\CB^{s+1}}}.
\]
Applying $(\Id-\CK)^{-1}$ to \eqref{eq:IdK} yields the existence and uniqueness of the solution $u$, and 
$$
\|u\|_{\CB^{s+2}}
\le \|(\Id-\CK)^{-1}\|\,\|\CL_0^{-1}f\|_{\CB^{s+2}}
\le
\frac{\min\{\alpha,m\}}{\min\{\alpha,m\}-\|E\|_{\CB^{s+1}}}
\cdot
\frac{2}{\min\{\alpha,m\}}\|f\|_{\CB^s}
=
L\|f\|_{\CB^s}.
$$

\section{Examples of Dimension-Independent Approximation}
\label{app:mainremark}
In this appendix, we present an example to illustrate Corollary~\ref{cor:maincor}. Let
\[
\CL u \coloneqq -\nabla \cdot (A(x)\nabla u)
+ b(x)\cdot \nabla u + c(x)u,
\]
and let $k\in \BN$, $\alpha,m>0$, and $K,F\ge 0$ be fixed constants.
Denote by $P(k,\alpha,m,K,F)$ the class of PDEs of the form $\CL u = f$ on $\BR^d$ satisfying the following conditions:
\begin{enumerate}[(1)]
    \item The spatial dimension $d\in\BN_{+}$ is arbitrary;
    \item the coefficients of $\CL$ satisfy (A1), (A2), and (A3') in Assumption~\ref{assum} with  $s=k$ and constants $\alpha$, $m$, and $K$;
    \item the source term satisfies $\|f\|_{\CB^k}\le F$.
\end{enumerate}
Assume additionally that the volume of the domain $\Omega\subseteq \mathbb{R}^d$ satisfies
$|\Omega|\le V$ for some constant $V>0$ independent of the dimension.
Then Corollary~\ref{cor:maincor} implies that, for any PDE in
$P(k,\alpha,m,K,F)$, the unique solution
$u^*\in \CB^{k+2}$ can be approximated (with error $\leq \varepsilon $) in the $H^{k+2}(\Omega)$ norm
by a two-layer neural network of the form~\eqref{eq:neural}
with at most
\[
n=\lceil C'\varepsilon^{-2}\rceil
\]
neurons, where
\[
C'
\le
\left(
\frac{2(\min\{\alpha,m\}+K)}
{\min\{\alpha,m\}(\min\{\alpha,m\}-K)}
\right)^2
V F^2.
\]
In particular, the bound on $n$ does not depend on the spatial dimension $d$.

A more precise example (in the Introduction) is given below. Let 
\[
k=0, \quad \alpha=m=1, \quad K=\frac{1}{2},\quad F=1,
\]
and let $\Omega=[0,1]^d$, so that $|\Omega|=1$ and hence $V=1$.
The condition $F=1$ is attainable. For example, the standard Gaussian density function in an arbitrary dimension,
\[
\varphi(x)=\frac{1}{(2\pi)^{d/2}}e^{-|x|^2/2}
\]
satisfies
\[
\|\varphi\|_{\CB^0}=\frac{1}{(2\pi)^{d/2}}<1.
\]
Corollary~\ref{cor:maincor} then implies that, for every PDE in $P(0,1,1,1/2,1)$, for instance,
\[
-\nabla \cdot (A(x)\nabla u) + b(x) \cdot \nabla u + c(x)u=\varphi,\quad \|A-I_d\|_{\mathcal{B}^{1}}+\|b-\mathbf{1}\|_{\mathcal{B}^0}+\|c-1\|_{\mathcal{B}^0}\leq \frac{1}{2},
\]
the unique solution
$u^*\in \CB^{2}$ can be approximated (with error $\leq \varepsilon $) by a two-layer neural network
of the form~\eqref{eq:neural} in the $H^{2}(\Omega)$ norm with at most
\[
n=\lceil 36 \varepsilon^{-2}\rceil
\]
neurons.

\addcontentsline{toc}{section}{References}
\bibliographystyle{alpha}
\bibliography{reference.bib}


\vspace{1em}

\noindent 
\small{\textsc{(ZC) Department of Mathematics, Massachusetts Institute of Technology, 77 Massachusetts Avenue, Cambridge, MA 02139, USA}}

\noindent 
\small{\textit{Email address}: \texttt{ziang@mit.edu}}

\vspace{1em}

\noindent 
\small{\textsc{(LH) Department of Mathematics and Statistics, Boston University, 665 Commonwealth Avenue, Boston, MA 02215, USA}}

\noindent 
\small{\textit{Email address}: \texttt{lqhuang@bu.edu}}

\vspace{1em}

\noindent 

\noindent
\small{(MY) \textsc{Department of Mathematics, Texas A\&M University, College Station, TX 77843, USA}}

\noindent 
\small{\textit{Email address}: \texttt{yangmx@tamu.edu}}

\vspace{1em}

\noindent 
\small{\textsc{(SZ) Institut de Math{\'e}matiques de Toulouse, 118 route de Narbonne, F-31062 Toulouse Cedex 9, France}}

\noindent 
\small{\textit{Email address}: \texttt{zhoushx98@outlook.com}}


\end{document}